\documentclass[a4paper,11pt]{article}
\title{
Enhancing Top Efficiency by Minimizing Second-Best Scores: \\
A Novel Perspective on Super Efficiency Models in DEA
}
\author{
Tomonari Kitahara\footnote{Faculty of Economics, Kyushu University, Fukuoka 819-0395, Japan. E-mail: tomonari.kitahara@econ.kyushu-u.ac.jp}      
~and 
Takashi Tsuchiya\footnote{Graduate School of Policy Studies, National Graduate Research Institute for Policy Studies, Tokyo 106-8677, Japan. 
E-mail: tsuchiya@grips.ac.jp}
}
\date{
November 1, 2024
}
\usepackage[margin=25mm]{geometry}
\usepackage{amsmath,amssymb,amsfonts,amsthm}
\usepackage{graphicx}
\usepackage{multirow}
\newtheorem{lemma}{Lemma}[section]

\def\bmath#1{\mbox{\boldmath $#1$}}

\def\t{\bmath{t}}
\def\u{\bmath{u}}
\def\v{\bmath{v}}

\def\x{\bmath{x}}
\def\y{\bmath{y}}

\def\0{\bmath{0}}
\def\1{\bmath{1}}
\begin{document}
\maketitle
\begin{abstract}
In this paper, we reveal a new characterization of the super-efficiency model for Data Envelopment Analysis (DEA).
In DEA, the efficiency of each decision making unit (DMU) is measured by the ratio 
the weighted sum of outputs divided by the weighted sum of inputs.
In order to measure efficiency of a DMU, ${\rm DMU}_j$, say, in CCR model, the
weights of inputs and outputs are determined so that the effiency of ${\rm DMU}_j$ 
is maximized under the constraint that the efficiency of each DMU is less than or equal to one. 
${\rm DMU}_j$ is called CCR-efficient if its efficiency score is equal to one.
It often happens that weights making ${\rm DMU}_j$ CCR-efficient are not unique but form 
continuous set.  
This can be problematic because the weights representing CCR-efficiencty of ${\rm DMU}_j$
play an important role in making decisions on its management strategy.  
In order to resolve this problem, 
we propose to choose weights which minimize the efficency of the second best DMU
enhancing the strength of ${\rm DMU}_j$, and demonstrate that this problem 
is reduced to a linear programming problem
identical to the renowned super-efficiency model.  
We conduct numerical experiments using data of Japanese commercial banks
to demonstrate the advantage of the supper-efficiency model.
\end{abstract}
\section{Introduction}
Data Envelopment Anlysis (DEA) is a widely used tool for comparing the efficiency of decision making units (DMUs). 
When each DMU has multiple inputs and outputs, the efficiency of the underlying DMU is measured by the weighted sum of inputs divided by the weighted sum of outputs. 
Under the classical CCR model developed by Charnes, Cooper and Rhodes \cite{ccr}, the weights of inputs and outputs are decided so that the effiency of the underlying DMU is maximized 
under the constraint that the efficiency of each DMU is less than or equal to one. 
A DMU whose efficiency is equal to one is said to be CCR-efficient. 
As a standard text of DEA, we list one by Cooper et al \cite{text}.
\\

For a CCR-efficient DMU, there may be multiple ways of choosing weights which give the efficiency equal to one. 
If this is the case, it is desirable to choose weights which highlight the differences between the underlying DMU and the other DMUs. 
Motivated by this idea, in this paper we consider to chose for weights which minimize the maximum efficiency of the other DMU, under the constraint that they give the efficiency equal to one for the underlying DMU.
Surprisingly, we show that the well-known super-efficiency model \cite{ap,dh,maj} exactly does this job.
Namely, the weight minimizing a score of the second-best DMU is obtained by solving the associated super-efficiency model.
\\

We conduct numerical experiments using data of Japanese commercial banks and observe that when applied to a CCR-efficient DMU, we can obtain weights which clearly reflect strengths and weaknesses of the DMU.
\section{CCR model}
Suppose there are $n$ decision making units (DMUs). 
For $i\in\{1,\dots,n\}$, $i$-th DMU has $l$ inputs $x_{1i},~\dots,x_{li}$ and $m$ outputs $y_{1i},\dots,y_{mi}$. From these inputs and outputs, we define the input vector
\[\x_i=(x_{1i},\dots,x_{li})^{\top}
\]
and the output vector
\[
\y_i=(y_{1i},\dots,y_{mi})^{\top}.
\]
We also define the matrix $X$ which is consisted of the input vectors as
\[
X=[\x_1,\x_2,\dots,\x_n]
\]
and the matrix $Y$ consisted of the output vectors
\[
Y=[\y_1,\y_2,\dots,\y_n].
\]
Further, we define matrix $X_{-i}$ as
\[
X_{-i}=[\x_1,\x_2,\dots,\x_{i-1},\x_{i+1},\dots,\x_n]
\]
that is , $X_{-i}$ is the matrix obtained by deleting the $i$-th column vector of $X$. 
Similarly, we define the matrix $Y_{-i}$.
\\

Suppose we want to measure the efficiency of $o$-th DMU ($o\in\{1,\dots,n\}$). In DEA, we measure the effiency of the $o$-th DMU by the ratio
\[
\frac{\y_o^{\top}\v}{\x_o^{\top}\u}
\]
where $\u\in\mathbb{R}^m_+$ and $\v\in\mathbb{R}^l_+$ are weight vectors.
Under CCR model, nonnegative weights $\u$ and $\v$ which maximize the efficiency of  the $o$-th DMU subject to the effiency of each DMU is not more than one, is sought. 
To be precise, the problem
\begin{equation}
\begin{array}{ll}
\max&\theta_o=\frac{\y_o^{\top}\v}{\x_o^{\top}\u}\\
{\rm subject~to}& \frac{\y_i^{\top}\v}{\x_i^{\top}\u}\le 1,~i\in\{1,\dots,n\}\ \\
&\u\ge\0,~\v\ge\0
\label{ccr_pure}
\end{array}
\end{equation}
is solved to find nonegative weights $\u$ and $\v$. This model is proposed by Charnes, Cooper and Rhodes in 1978.\\ 

Note that the ratio $\frac{\y_i^{\top}\v}{\x_i^{\top}\u}$ is unchanged if we multiply $\u$ and $\v$ by some constant. Thus we can assume $\x_o^{\top}\u=1$. Then we can rewrite (\ref{ccr_pure}) to the following linear programming problem (LP).
\[
\begin{array}{ll}
\max&\theta_o=\y_o^{\top}\v\\
{\rm subject~to}&\x_o^{\top}\u=1\\
& \y_i^{\top}\v\le\x_i^{\top}\u,~i\in\{1,\dots,n\}\ \\
&\u\ge\0,~\v\ge\0
\end{array}
\]
This problem can be expressed compactly by using the matrix $X$ and $Y$ as
\begin{equation}
\begin{array}{ll}
\max&\theta_o=\y_o^{\top}\v\\
{\rm subject~to}&\x_o^{\top}\u=1\\
& Y^T\v\le X^T\u\\
&\u\ge\0,~\v\ge\0
\label{ccr_lp}
\end{array}
\end{equation}

The optimal value $\theta_o^*$ of (\ref{ccr_lp}) is called CCR-efficient. 
If $\theta_o^*=1$, $o$-th DMU is said to be CCR-efficient, otherwise it is said to be 
CCR-inefficient. 
\section{
Minimizaing the efficiency score of the second best DMU
}
Assume $o$-th DMU is CCR-efficient, namely $\theta_o^*=1$. 
In this case there could be multiple $(\u,\v)$ which is a feasible solution of (\ref{ccr_lp}) 
and $\y_o^{\top}\v=1$. 
Out of these weight vectors, we want to choose one which minimizes the maximum value of the efficiencies of the other DMUs. 
By choosing a weight in this way, we expect that we can make the difference between $o$-th DMU (i.e. a CCR-efficient DMU) 
and the other DMUs more clearly. 
This kind of weights can be obtained by solving the following problem.
\begin{equation}
\begin{array}{ll}
\min&t_o\\
{\rm subject~to}&\frac{\y_o^{\top}\v}{\x_o^{\top}\u}=1\\
& \frac{\y_i^{\top}\v}{\x_i^{\top}\u}\le t_o,~i\in\{1,\dots,n\}\setminus\{o\} \\
&\u\ge\0,~\v\ge\0
\label{new_pure}
\end{array}
\end{equation}
As we explained above, we can assume that $\x_o^{\top}\u=1$. 
Then (\ref{new_pure}) becomes
\[
\begin{array}{ll}
\min&t_o\\
{\rm subject~to}&\x_o^{\top}\u=1\\
&\y_o^{\top}\v=1\\
& \y_i^{\top}\v\le t_o\x_i^{\top}\u,~i\in\{1,\dots,n\}\setminus\{o\} \\
&\u\ge\0,~\v\ge\0
\end{array}
\]
By setting $\tilde{\u}=t_o\u$, this problem can be reduced to the following LP. 
\[
\begin{array}{ll}
\min&t_o\\
{\rm subject~to}&\x_o^{\top}\tilde{\u}=t_o\\
&\y_o^{\top}\v=1\\
& \y_i^{\top}\v\le \x_i^{\top}\tilde{\u},~i\in\{1,\dots,n\}\setminus\{o\} \\
&\u\ge\0,~\tilde{\v}\ge\0
\end{array}
\]
By using the matrices $X_{-o}$ and $Y_{-o}$ defined in Section 2, we can express this problem as
\begin{equation}
\begin{array}{ll}
\min&t_o\\
{\rm subject~to}&\x_o^{\top}\tilde{\u}=t_o\\
&\y_o^{\top}\v=1\\
&Y_{-o}^T\v\le X_{-o}^T\tilde{\u}\\
&\tilde{\u}\ge\0,~\v\ge\0
\label{new_lp}
\end{array}
\end{equation} 

We can think of the problem (\ref{new_lp}) for a DMU which is not CCR-efficient. 
We have the following lemma for the problem (\ref{new_lp}).
\begin{lemma} 
Let $t_o^*$ be the optimal value for the problem (\ref{new_lp}). 
Then $o$-th DMU is CCR-efficient if and only if $t_o^*\le 1$.
\end{lemma}
\begin{proof}
First assume $o$-th DMU is CCR-efficient. Then there exist $\u^*$ and $\v^*$ satifying
\[
\begin{array}{l}
\x_o^{\top}\u^*=1\\
\y_o^{\top}\v^*=1\\
Y^T\v^*\le X^T\u^*\\
\u\*\ge\0,~\v^*\ge\0
\end{array}
\]
Then $\u^*$ and $\v^*$ together with $t_o=1$ is a feasible solution of (\ref{new_lp}) with the objective value 1.
Thus $\t_o^*\le 1$.

Next let ($\tilde{\u}^*,\v^*,t_o^*)$ be an optimal solution of (\ref{new_lp}) with $t_o^*\le 1$. 
Then it is easy to see that $(\frac{\tilde{\u}^*}{t_o^*},\v^*)$ is a feasible solution of (\ref{ccr_lp}) with the objective value 1.
Thus $o$-th DMU is CCR-efficient.
\end{proof}

Using standard techniques of linear programming, we can show that the dual problem of (\ref{new_lp}) is given as follows.
\begin{equation}
\begin{array}{lll}
(S)&\max& \tilde{\eta} \\
&{\rm subject~to}&X\lambda_{-o}\le \x_o\\
&&Y\lambda_{-o}\ge \tilde{\eta}\y_o\\
&&\lambda_{-o}\ge\0 \label{d-new}
\end{array}
\end{equation} 
where $\tilde{\eta}\in\mathbb{R}$ and $\lambda_{-o}\in\mathbb{R}^{n-1}$ are variables.
In the dual problem we try to make inputs and outputs using those of the DMUs other than $o$-th DMU and nonnegative weights, 
so that the resulting inputs are less than or equal to that of $o$-th DMU, and the resulting output is $\tilde{\eta}$ times bigger 
than or equal to that of $o$-th DMU.
Note that this is an output-oriented model, as we try to adjust outputs of the $o$-th DMU. 
We can see that in the dual, the $o$-th unit is not included in the reference set, and the model (S) is identical to the renowned super-efficiency model \cite{ap}, 
where its objective is to rank efficient DMUs. 
We note that Mehrabian et al. \cite{maj} propose an alternative approach which overcomes some drawbacks of the original super efficiency model. 
We also remark that effects of excluding the column being scored from the input and output matrices are studied by Dul\'{a} and Hickman \cite{dh}.
%
\section{Relations between CCR and the super-efficiency model}
In this section, we investigate relations between the classical CCR model and the super-efficiency model.
First, the dual of CCR model (\ref{ccr_lp}) is given as follows.
\[
\begin{array}{ll}
\min& \theta \\
{\rm subject~to}&\theta \x_0\ge X\lambda'\\
&\y_0\le Y\lambda'\\
&\lambda' \ge\0
\end{array}
\]
where $\theta\in\mathbb{R}$ and $\lambda'\in\mathbb{R}^n$ are variables. 
Similar to (S), we make the model output-oriented 
by setting $\eta=\frac{1}{\theta}$ and $\lambda=\lambda'/\theta$, and converting the model to
\begin{equation}
\begin{array}{lll}
(C)&\max& \eta \\
&{\rm subject~to}&\x_0\ge X\lambda\\
&&\eta\y_0\le Y\lambda\\
&&\lambda\ge\0 \label{d-ccr}
\end{array}
\end{equation} 
In association with (\ref{d-ccr}), we define the following slack-maximization problem.
\begin{equation}
\begin{array}{lll}
(CSM)&\max&\1^T\epsilon_{1}+\1^T\epsilon_{2}\\
&{\rm subject~to}&\x_0-\epsilon_1=X\lambda\\
&&\eta^*\y_0+\epsilon_2= Y\lambda\\
&&\lambda\ge\0,~\epsilon_1\ge\0,~\epsilon_2\ge\0 \label{csm}
\end{array}
\end{equation} 
where $\eta^*$ is the optimal value of (C) and $\lambda\in\mathbb{R}^n,~\epsilon_1\in\mathbb{R}^l$, and $\epsilon_2\in\mathbb{R}^m$ are variables. 
Similarly, for the problem (\ref{d-new}), we define the slack-maximization problem.
\begin{equation}
\begin{array}{lll}
(SSM)&\max&\1^T\epsilon_{1}+\1^T\epsilon_{2}\\
&{\rm subject~to}&\x_0-\epsilon_1=X_{-o}\lambda_{-o}\\
&&\tilde{\eta}^*\y_0+\epsilon_2=Y_{-o}\lambda_{-o}\\
&&\lambda_{-o}\ge\0,~\epsilon_1\ge\0,~\epsilon_2\ge\0 \label{nsm}
\end{array}
\end{equation} 
where $\tilde{\eta}^*$ is the optimal value of (S) and $\lambda_{-o}\in\mathbb{R}^{n-1},~\epsilon_1\in\mathbb{R}^l$ and $\epsilon_2\in\mathbb{R}^m$ are variables. \\

Relations between the classical CCR model and the super-efficiency model are summarized in a $4\times 4$ table Table \ref{t_r}. 
In the table, ``Opt" means ``Optimal value".
Table \ref{t_r} is interpreted as follows. For example, if the optimal value of (CSM) is equal to zero and (CSM) has a unique solution $(\lambda_o,\lambda_{-o})=(1,\0)$ (see the 
first row of the table), 
then the optimal value of (S) is less than one and this is the only possibility, and so on.
\begin{table}[hbtp]
  \caption{Relations between CCR model and the super-effiency model}
  \label{t_r}
  \centering
  \scalebox{0.6}{
  \begin{tabular}{|c|c|c|c|c|c|c|}
   \hline
  &&&Opt of (S)$>1$&\multicolumn{2}{|c|}{Opt of (S)$=1$}&Opt of (S)$<1$\\ \hline
&&&&Opt of (SSM) $>0$&Opt of (SSM) $=0$& \\ \hline
\multirow{3}{8em}{Opt of (C)$=1$}&\multirow{2}{8em}{Opt of (CSM) $=0$}&Unique optimal solution $(\lambda_0,\lambda_{-o})=(1,\0)$&$\times$ (R0)&$\times$ (R1)&$\times$ (R2)& $\bigcirc$ \\ \cline{3-7}
&&Other than the above&$\times$ (R0)&$\times$ (R3)&$\bigcirc$&$\times$ (R4) \\ \cline{2-7}
&\multicolumn{2}{|c|}{Opt of (CSM) $>0$}&$\times$ (R0)&$\bigcirc$&$\times$ (R5)&$\times$ (R6) \\ \hline
\multicolumn{3}{|c|}{Opt of (C)$>1$}&$\bigcirc$&$\times$   (R0)&$\times$ (R0)&$\times$ (R0)\\ \hline
  \end{tabular}
}
\end{table}
\\

{\bf Proof of (R0)} As we obeserved in Lemma 3.1, $o$-th DMU is efficient if and only if Opt of (S) is less than or equal to one. 
Similarly it is easy to see that  $o$-th DMU is efficient if and only if Opt of (C) is equal to one. 
By combining these facts, we have (R0).\\

{\bf Proof of (R1), (R3)} Assume the optimal value of (CSM) is equal to 0 and (CSM) has a unique optimal solution 
$(\lambda_o,\lambda_{-o})=(1,\0)$.
Assume also that the optimal value of (S) is equal to one and the optimal value of (SSM) is positive. 
Then by considering an optimal solution of (S), there exsit $\lambda_{-o}\in\mathbb{R}^{n-1},~\epsilon_1\in\mathbb{R}^l$ and $\epsilon\in\mathbb{R}^m$
satisfying the following system.
\[
\begin{array}{l}
\1^T\epsilon_1+\1^T\epsilon_2>0\\
\x_o-\epsilon_1=X_{-o}\lambda_{-o}\\
\y_o+\epsilon_2=Y_{-o}\lambda_{-o}\\
\lambda_{-0}\ge\0,~\epsilon_1\ge\0,~\epsilon_2\ge\0
\end{array}
\]
Then $(0,\lambda_{-o}), \epsilon_1$ and $\epsilon_2$ is a feasible solution of (CSM), which has a positive objective value. 
This contradicts our assumption that the optimal value of (CSM) is equal to zero. 
Thus (R1) is proved. We can show (R3) similarly.\\

{\bf Proof of (R2)} Assume the optimal value of (CSM) is equal to 0 and (CSM) has a unique optimal solution 
$(\lambda_o,\lambda_{-o})=(1,\0)$. 
Assume also that the optimal value of (S) is equal to 1 and the optimal value of (SSM) is equal to zero. 
Then by considering an optimal solution of (SSM), there exist 
$\lambda_{-o}\in\mathbb{R}^{n-1},~\epsilon_1\in\mathbb{R}^l$ and $\epsilon\in\mathbb{R}^m$
satisfying the following system.
\[
\begin{array}{l}
\1^T\epsilon_1+\1^T\epsilon_2=0\\
\x_o-\epsilon_1=X_{-o}\lambda_{-o}\\
\y_o+\epsilon=Y_{-o}\lambda_{-o}\\
\lambda_{-0}\ge\0,~\epsilon_1\ge\0,~\epsilon_2\ge\0
\end{array}
\]
Then $(0,\lambda_{-o}),~\epsilon_1$ and $\epsilon_2$ is a feasible solution of (CSM), whose objective is equal to zero, namely the optimal value of (CSM). 
Thus this solution is an optimal solution of (SSM), which is supposed to have a unique optimal solution $(\lambda_o,\lambda_{-o})=(1,\0)$, a contradiction. \\

{\bf Proof of (R4)} Assume that the optimal value of (CSM) is equal to zero and it has an optimal solution $(\lambda_o,\lambda_{-o})\not=(1,\0)$ 
and the optimal value of (S) is less than 1,
Then by considering an optimal solution of (CSM), there exist $(\lambda_o,\lambda_{-o})$ satisfying the following system.
\[
\begin{array}{l}
\x_o=\lambda_o\x_o+X_{-o}\lambda_{-o}\\
\y_o=\lambda_o\y_o+Y_{-o}\lambda_{-o}
\end{array}
\]
If $\lambda_o=1$, then we have $\lambda_{-o}=\0$, contradicting our assumption. 
Thus $\lambda_0\not=1$ and in particular, $\lambda<1$. 
Then by a simple calculation leads to 
\[
\begin{array}{l}
\x_o=X_{-o}\left(\frac{1}{1-\lambda}\lambda_{-o}\right)\\
\y_o=Y_{-o}\left(\frac{1}{1-\lambda}\lambda_{-o}\right)\\
\end{array}
\]
Then $\eta=1,~\frac{1}{1-\lambda}\lambda_{-o}$ is a feasible solution of (S) with the objective value equal to 1. 
This is a contradiction since we assumed the optimal value of (S) is less than 1. \\

{\bf Proof of (R5)} Assume that the optimal values of (CSM), (S) and (SSM) are positive, one and zero, respectively. 
Then by considering an optimal solution of (CSM), there exist $(\lambda_o,\lambda_{-o}), \epsilon_1$ and $\epsilon_2$ satisfying the following system.
 \[
\begin{array}{l}
\1^T\epsilon_1+\1^T\epsilon_2>0\\
\x_o-\epsilon_1=\lambda_o\x_o+X_{-o}\lambda_{-o}\\
\y_o+\epsilon_2=\lambda_o\y_o+Y_{-o}\lambda_{-o}\\
(\lambda_o,\lambda_{-o})\ge\0\\
\epsilon_1\ge\0,~\epsilon_2\ge\0
\end{array}
\]
If $\lambda_o=1$, then we have $\lambda_o=\0$ from the second equation. 
Then we have $\epsilon_1=\0$ and $\epsilon_2=\0$, which contradict the first equation. 
Thus we have $\lambda_o\not=1$ and in particular, $\lambda_o<1$. 
Then form the second and the third equation, we have 
 \[
\begin{array}{l}
\x_o-\frac{1}{1-\lambda_0}\epsilon_1=X_{-o}\left(\frac{1}{1-\lambda}\lambda_{-o}\right)\\
\y_o+\frac{1}{1-\lambda_0}\epsilon_2=Y_{-o}\left(\frac{1}{1-\lambda}\lambda_{-o}\right) \label{eq}
\end{array}
\]
By noting the $o$-th DMU is CCR-efficient, i.e. $\tilde{\eta}^*=1$, $\frac{1}{1-\lambda}\lambda_{-o}, {1-\lambda_0}\epsilon_1$ and 
${1-\lambda_0}\epsilon_2$ is a feasible solution of (SSM) whose objective value is positive. 
This is a contradiction since we assumed that the optimal value of (SSM) is equal to 0.\\

{\bf Proof of (R6)}  Assume that the optimal values of (CSM) and (S) are positive and less than 1, respectively. 
Then similarly as the proof of (R5), there exist $(\lambda_o,\lambda_{-o}), \epsilon_1$ and $\epsilon_2$
satisfying (\ref{eq}). 
Then $\eta=1$ and $\frac{1}{1-\lambda}\lambda_{-o}$ is a feasible solution of (S) with the objective value equal to one. 
But this is a contradiction since we assumed that the optimal value of (S) is less than 1.
\section{Numerical example}
In this section, we show a numerical example using real data. 
In this example, we compare the efficiency of  21 Japanese banks (4 city banks (C), 14 regional banks (R) and 3 others (O)) in 2016. 
The data used in this section are from \cite{mt}, and they were originally from financial statements of the banks.  Following \cite{mt}, we adopt interest expenses and non-interest expenses as inputs and we select interest income and non-interest income as outputs. 
We show the data in Table \ref{t1}.

\begin{table}[hbtp]
  \caption{Inputs and outputs for 18 Japanese banks (unit: Yen)}
  \label{t1}
  \centering
  \scalebox{0.9}{
  \begin{tabular}{|c|l|c|r|r|r|r|}
    \hline
    No.&Name&Type of bank&\multicolumn{2}{c}{Inputs}&\multicolumn{2}{c}{Outputs}\\ \hline
        &        &                  &Interest&Non-interest&Interest& Non-interest\\ \hline
1 & Mizuho Financial Group & C & 577,737 & 1,977,650 & 1,445,555 & 1,847,345 \\
2 & Sumitomo Mitsui Banking Corporation & C & 553,394 & 3,573,995 & 1,912,027 & 3,221,218 \\
3 & Mitsubishi UFJ Financial Group & C & 863,677 & 3,755,124 & 2,888,134 & 3,091,434 \\
4 & Resona Bank & C & 28,422 & 505,224 & 406,328 & 355,529 \\
5 & Aozora Bank & O  & 21,507 & 61,432 & 67,154 & 67,550 \\
6 & Shinsei Bank & O  & 16,209 & 167,957 & 138,488 & 122,587 \\
7 & Japan Post Bank Co & O  & 348,746 & 1,107,938 & 1,567,512 & 329,769 \\
8 & The Chiba Bank & R & 16,589 & 133,618 & 135,533 & 92,278 \\
9 & The Bank of Yokohama & R & 10,953 & 222,692 & 183,217 & 206,953 \\
10 & The Shizuoka Bank & R & 14,661 & 188,335 & 123,005 & 126,799 \\
11 & The Bank of Fukuoka & R & 15,988 & 97,002 & 123,899 & 48,873 \\
12 & Hokuhoku Financial Group  &R & 6,243 & 141,699 & 120,786 & 66,634 \\
13 & Hokuyo Bank (North Pacific Bank) & R & 3,471 & 123,104 & 78,229 & 69,743 \\
14 & Aichi Bank & R & 1,282 & 41,188 & 31,015 & 19,016 \\
15 & Miyazaki Bank & R & 2,014 & 35,993 & 34,558 & 19,371 \\
16 & Ehime Bank & R & 2,861 & 31,728 & 33,120 & 8,943 \\
17 & Bank of Kyoto & R & 5,082 & 77,697 & 70,725 & 39,755 \\
18 & Hiroshima Bank & R & 9,417 & 83,760 & 80,579 & 57,684 \\
19 & Tottori Bank & R & 996 & 13,246 & 12,112 & 4,080 \\
20 & Akita Bank & R & 2,709 & 38,369 & 31,235 & 16,230 \\
21 & The Bank of Iwate & R & 1,486 & 36,986 & 31,863 & 19,268 \\ \hline
  \end{tabular}
}
\ \\
\ \\

  \caption{The effiency of banks (CCR-efficient banks are in bold letters.)}
  \label{t2}
  \centering
  \scalebox{0.80}{
  \begin{tabular}{|c|l|c|c|c|c|c|}
    \hline
    No.&Name&CCR-Efficiency&\multicolumn{2}{|c|}{Input weight}&\multicolumn{2}{|c|}{Output weight}\\ \hline
	&	&	       &$u_1$&$u_2$&$v_1$&$v_2$\\ \hline
1 & Mizuho Financial Group & 0.876 &   $2.68\times10^{-7}$ & $4.27\times10^{-7}$& 0 & $4.74\times10^{-7}$\\
2 & Sumitomo Mitsui Banking Corporation & 0.911 &   $1.60\times10^{-7}$ & $2.55\times10^{-7}$ & 0 & $2.83\times10^{-7}$ \\
3 & Mitsubishi UFJ Financial Group & 0.798 &   $1.46\times10^{-7}$& $2.33\times10^{-7}$ & 0&$2.58\times10^{-7}$ \\
4 & Resona Bank & 0.905 &   $4.96\times10^{-6}$ & $1.70\times10^{-6}$ & $1.76\times10^{-6}$ & $5.33\times10^{-7}$\\
5 & {\bf Aozora Bank} & {\bf 1.000} &  0 & $1.63\times10^{-5}$ & $1.06\times10^{-5}$ & $4.26\times10^{-6}$\\
6 & Shinsei Bank & 0.869 &   $5.01\times10^{-6}$ & $5.47\times10^{-6}$ & $3.85\times10^{-6}$ & $2.75\times10^{-6}$\\
7 & {\bf Japan Post Bank Co} & {\bf 1.000}  & 0 & $9.03\times10^{-7}$ & $6.38\times10^{-7}$ & 0\\
8 & The Chiba Bank & 0.952 &   $6.15\times10^{-6}$ & $6.72\times10^{-6}$ & $4.72\times10^{-6}$ & $3.37\times10^{-6}$ \\
9 & {\bf The Bank of Yokohama} & {\bf 1.000}  & $8.35\times10^{-5}$ & $3.83\times10^{-7}$ & 0 & $4.83\times10^{-6}$ \\
10 & The Shizuoka Bank & 0.744 &  $4.54\times10^{-6}$ & $4.96\times10^{-6}$ & $3.48\times10^{-6}$ & $2.49\times10^{-6}$ \\
11 & {\bf The Bank of Fukuoka} & {\bf 1.000} & $2.35\times10^{-5}$ & $6.43\times10^{-6}$ & $8.07\times10^{-6}$ & 0 \\
12 & Hokuhoku Financial Group & 0.962 & $4.98\times10^{-5}$ & $4.86\times10^{-6}$ & $7.97\times10^{-6}$ & 0 \\
13 & {\bf Hokuyo Bank (North Pacific Bank)} & {\bf 1.000}   & $2.88\times10^{-4}$ & 0 & $9.98\times10^{-6}$ & $3.14\times10^{-6}$ \\
14 & Aichi Bank & {\bf 1.000} &  $4.52\times10^{-4}$ & $1.02\times10^{-5}$ & $2.37\times10^{-5}$ & $1.39\times10^{-5}$ \\
15 & {\bf Miyazaki Bank} & {\bf 1.000} & $6.97\times10^{-5}$ & $2.39\times10^{-5}$ & $2.47\times10^{-5}$ & $7.48\times10^{-6}$ \\
16 & Ehime Bank & 0.985 & $8.66\times10^{-5}$ & $2.37\times10^{-5}$ & $2.97\times10^{-5}$ & 0 \\
17 & Bank of Kyoto & 0.926 & $3.15\times10^{-5}$ & $1.08\times10^{-5}$ & $1.12\times10^{-5}$ & $3.39\times10^{-6}$ \\
18 & Hiroshima Bank & 0.927 & $9.91\times10^{-6}$ & $1.08\times10^{-5}$ & $7.61\times10^{-6}$ & $5.44\times10^{-6}$ \\
19 & Tottori Bank & 0.900 & $2.16\times10^{-4}$ & $5.92\times10^{-5}$ & $7.43\times10^{-5}$ & 0 \\
20 & Akita Bank & 0.812 &  $7.57\times10^{-5}$ & $2.07\times10^{-5}$ & $2.60\times10^{-5}$ & 0 \\
21 & {\bf The Bank of Iwate} & {\bf 1.000}  & $3.69\times10^{-4}$ & $1.22\times10^{-5}$ & $2.50\times10^{-5}$ & $1.05\times10^{-5}$ \\ \hline
  \end{tabular}
}
\end{table}
First we check the efficiency of each bank (DMU) by CCR model solving (\ref{ccr_lp}). Results are summarized in Table \ref{t2}. 
CCR-efficient banks are indicated in bold letters.
We see that there are 7 CCR-efficient banks whose efficiency values are 1.

Now, we choose one particular CCR-effient bank, namely the Bank of Yokohama
to feature the super-efficiency model in comparison with CCR model. 
From Table \ref{t2}, we see that 
the input and weight weights $(\u_{ccr}^*,\v_{ccr}^*)$ by CCR model for the Bank
of Yokohama are
\begin{equation} \label{weightCCR}
\u_{ccr}^*=(8.35\times10^{-5},3.83\times10^{-7})^{\top}\ \hbox{and}\ \v_{ccr}^*=(0,4.83\times10^{-6})^{\top}.
\end{equation}   
The efficiency of each bank based on 
this weight is shown in the first column of Table 4.  We see that there is 
another efficient bank Hokuyo bank under this weight. 

Next we solve (\ref{new_lp}) with $o=9$ (9 is ID of the Bank of Yokohama)
to compute the input and output weights by the super-efficiency model and obtain 
that
$t^*=0.720,~\tilde{\u}_{s}^*=(2.36\times10^{-5},2.07\times10^{-6})^{\top},~\v_{s}^*=(0,4.83\times10^{-6})^{\top}$. 
Dividing $\tilde{\u}_s$ by $t^*$ to scale the weight so that the efficiency of 
the most efficient unit, namely, the Bank of Yokohama is 1, we obtain the weights 
\begin{equation}\label{weightSE}
\tilde{\u}_{s}/t^*=(3.28\times10^{-5},2.88\times10^{-6})^{\top}\ \hbox{and}\ 
\v_{s}^*=(0,4.83\times10^{-6})^{\top}
\end{equation}
by the super-efficient model.  
The efficiency of each bank based on this weight is shown in the second column of Table \ref{t3}.
%
\begin{table}[hbtp]
  \caption{Comparison of the super-efficiency model and CCR model}
  \label{t3}
  \centering
  \scalebox{0.85}{
  \begin{tabular}{|c|l|c|c|}
    \hline
    No.&Name&Efficiency (CCR model) &Efficiency (super-efficiency model) \\ \hline
1 & Mizuho Financial Group & 0.182 & 0.362 \\
2 & Sumitomo Mitsui Banking Corporation & 0.327 & 0.547 \\
3 & Mitsubishi UFJ Financial Group & 0.203 & 0.382 \\
4 & Resona Bank & 0.669 & 0.720 \\
5 & Aozora Bank & 0.179 & 0.370 \\
6 & Shinsei Bank & 0.418 & 0.584 \\
7 & Japan Post Bank Co & 0.054 & 0.109 \\
8 & The Chiba Bank & 0.310 & 0.480 \\
9 & The Bank of Yokohama & 1.000 & 1.000 \\
10 & The Shizuoka Bank & 0.473 & 0.599 \\
11 & The Bank of Fukuoka & 0.172 & 0.294 \\
12 & Hokuhoku Financial Group & 0.559 & 0.526 \\
13 & Hokuyo Bank (North Pacific Bank) & 1.000 & 0.720 \\
14 & Aichi Bank & 0.748 & 0.572 \\
15 & Miyazaki Bank & 0.514 & 0.552 \\
16 & Ehime Bank & 0.172 & 0.233 \\
17 & Bank of Kyoto & 0.423 & 0.492 \\
18 & Hiroshima Bank & 0.341 & 0.507 \\
19 & Tottori Bank & 0.223 & 0.279 \\
20 & Akita Bank & 0.325 & 0.394 \\
21 & The Bank of Iwate & 0.673 & 0.600 \\
\hline
  \end{tabular}
}
\end{table}

It is seen that,
while we have another efficient bank
Hokuyo Bank (North Pacific Bank) other than the Bank of Yokohama in CCR model, 
the Bank of Yokohama bank is the only efficient bank in the super-efficiency model, where efficiency of all other 
banks are lower than 0.720 with the second best being attained by Hokuyo Bank (North Pacific Bank) and Resona Bank.  
Thus, the super-efficiency model succeeds better in enhancing 
the strength of Yokohama bank.

By comparing the weights \eqref{weightCCR} of CCR model and \eqref{weightSE} of the super-efficiency model, 
we observe that there exists considerable 
difference in the input weights although the output weights are identical. 
In paticular, CCR model puts much higher weight on interest expenses than 
the super-efficiency model.  This suggests that
strength of the Bank of Yokohama lies in the balance between interest and non-interest expenses.
\section{Conclusion}
In this paper, we revealed a new characteristic of the super-efficiency model for DEA. 
In CCR model, weights of inputs and outputs are decided so that the effiency of the underlying DMU is maximized. 
For a CCR-efficient DMU, we considered to choose weights  so that the efficiency score of the second best DMU is minimized. 
We showed that the problem is reduced to a linear programming problem which
is identical to the dual of the famous super-efficiency model. 
We conducted numerical experiments and showed that we can obtain weights that highlight differences between efficient and non-efficient DMUs.
\section*{Acknowledgement}
The authors would like thank Mr.~Ulugbek Nizamov, a former student of the second author from the Republic of Uzbekistan, whose
master thesis \cite{mt} provided a motivation to this research. 
They also thank Prof.~Tomohiro Kinoshita of Department of Economics, Otemon Gakuin University 
(formerly National Graduate Institute for Policy Studies) for suggesting
the input and output variables of the DEA model of bank efficiency in Section 5.
The authors are also grateful to Prof. Kazuyuki Sekitani of Seikei University for pointing out connection of our analysis to super efficiency models \cite{ap}.
This work was supported by JSPS KAKENHI Grant Number JP21H03398.

%

\end{document}